\documentclass{proc-l-hijacked}
\usepackage{amssymb, amsmath, amsthm, hyperref,}

\newtheorem{theorem}{Theorem}[section]
\newtheorem{lemma}[theorem]{Lemma}

\newtheorem{corollary}[theorem]{Corollary}

\theoremstyle{definition}

\DeclareMathAlphabet{\mathpzc}{OT1}{pzc}{m}{it}

\numberwithin{equation}{section}

\begin{document}

	\title[Trace Characterizations and Socle Identifications]{Trace Characterizations and Socle Identifications in Banach Algebras}
	\author{G. Braatvedt, R. Brits \and F. Schulz}
	\address{Department of Mathematics, University of Johannesburg, South Africa}
	\email{gabraatvedt@uj.ac.za, rbrits@uj.ac.za,  francoiss@uj.ac.za }
	\subjclass[2010]{15A60, 46H05, 46H10, 46H15, 47B10 }
	\keywords{rank, socle, trace }

\begin{abstract}
As a follow-up to a paper of D. Petz and J. Zem\'{a}nek \cite{charactertracezem}, a number of equivalent conditions which characterize the trace among linear functionals on matrix algebras, finite rank operators and the socle elements of semisimple Banach algebras in general are given. Moreover, the converse problem is also addressed; that is, given the equivalence of certain conditions which characterize the trace, what can be said about the structure of the socle? In particular, we characterize those socles isomorphic to matrix algebras in this manner, as well as those socles which are minimal two-sided ideals.
\end{abstract}
	\parindent 0mm
	
	\maketitle

\section{Introduction}
The determinants of infinite matrices were first investigated by astronomers, over a century ago. Nowadays, the notions of rank, trace and determinant are well-established for operator theory. Recently, in their paper entitled \textit{Trace and determinant in Banach algebras} \cite{aupetitmoutontrace}, Aupetit and Mouton managed to show that these notions can be developed, without the use of operators, in a purely spectral and analytic manner. This paper is fundamental to our discussion here, for this alternative point of view not only permits the possibility to consider trace related problems in a more general setting, but also allows for the consideration of a converse problem to \cite[Theorem 2]{charactertracezem}. We briefly summarize some of the theory in \cite{aupetitmoutontrace} before we proceed.\\

By $A$ we denote a complex Banach algebra with identity element $\mathbf 1$ and invertible group $G(A)$. Moreover, it will be assumed throughout that $A$ is semisimple (i.e. the Jacobson Radical of $A$, denoted $\mathrm{Rad}\:A$, only contains $0$). For $x\in A$ we denote by $\sigma_A(x) =\left\{\lambda\in\mathbb C:\lambda\mathbf{1}-x\notin G(A)\right\}$, $\rho_{A} (x) = \sup\left\{\left|\lambda\right| : \lambda \in \sigma_{A} (x)\right\}$ and $\sigma_{A} '(x)=\sigma_A(x) -\{0\}$ the spectrum, spectral radius and nonzero spectrum of $x$, respectively. If the underlying algebra is clear from the context, then we shall agree to omit the subscript $A$ in the notation $\sigma_A(x)$, $\rho_{A} (x)$ and $\sigma_{A} '(x)$. This convention will also be followed in the forthcoming definitions of rank and trace. For each nonnegative integer $m$, let
$$\mathcal{F}_{m} = \left\{a \in A: \#\sigma'(xa) \leq m \;\,\mathrm{for\;all}\;\,x \in A \right\},$$
where the symbol $\#K$ denotes the number of distinct elements in a set $K\subseteq \mathbb C$. Following Aupetit and Mouton in \cite{aupetitmoutontrace}, we define the \textit{rank} of an element $a$ of $A$ as the smallest integer $m$ such that $a \in \mathcal{F}_{m}$, if it exists; otherwise the rank is infinite. In other words,
$$\mathrm{rank}\,(a) = \sup_{x \in A} \#\sigma'(xa).$$
If $a \in A$ is a finite-rank element, then
$$E(a) = \left\{x \in A : \#\sigma'(xa) = \mathrm{rank}\,(a) \right\}$$
is a dense open subset of $A$ \cite[Theorem 2.2]{aupetitmoutontrace}. A finite-rank element $a$ of $A$ is said to be a \textit{maximal finite-rank element} if $\mathrm{rank}\,(a) = \#\sigma'(a)$. With respect to $\mathrm{rank}$ it is further useful to know that $\sigma ' (xa)=\sigma '(ax)$ for all $x, a \in A$ (Jacobson's Lemma, \cite[Lemma 3.1.2.]{aupetit1991primer}). It can be shown \cite[Corollary 2.9]{aupetitmoutontrace} that the socle, written $\mathrm{Soc}\:A$, of a semisimple Banach algebra $A$ coincides with the collection $\bigcup_{m = 0}^{\infty} \mathcal{F}_{m}$ of finite rank elements. We mention a few elementary properties of the rank of an element \cite[p. 117]{aupetitmoutontrace}. Firstly, $\#\sigma'(a) \leq \mathrm{rank}\,(a)$ for all $a \in A$. Furthermore, $\mathrm{rank}\,(xa) \leq \mathrm{rank}\,(a)$ and $\mathrm{rank}\,(ax)\leq \mathrm{rank}\,(a)$ for all $x, a \in A$, with equality if $x \in G(A)$. Moreover, the rank is lower semicontinuous on $\mathrm{Soc}\:A$. It is also subadditive, i.e. $\mathrm{rank}\,(a+b) \leq \mathrm{rank}\,(a) + \mathrm{rank}\,(b)$ for all $a, b \in A$ \cite[Theorem 2.14]{aupetitmoutontrace}. Finally, if $p$ is a projection of $A$, then $p$ has rank one if and only if $p$ is a minimal projection, that is $pAp = \mathbb{C}p$. It is also worth mentioning here that a projection $p$ is minimal if and only if $Ap$ is a nontrivial left ideal which does not properly contain any left ideals other than $\left\{0 \right\}$, that is, if and only if $Ap$ is a nontrivial minimal left ideal \cite[Lemma 30.2]{bonsall1973complete}. A similar result holds true for the right ideal $pA$.\\

The following two results are fundamental to the theory developed in \cite{aupetitmoutontrace} and are mentioned here for convenient referencing later on:\\

\textbf{Scarcity Theorem for Rank} \cite[Theorem 2.3]{aupetitmoutontrace}: Let $f$ be an analytic function from a domain $D$ of $\mathbb{C}$ into $A$. Then either the set of $\lambda$ for which the rank of $f(\lambda)$ is finite has zero capacity or there exist an integer $N$ and a closed discrete subset $E$ of $D$ such that $\mathrm{rank}\,(f(\lambda)) = N$ on $D-E$ and $\mathrm{rank}\,(f(\lambda)) < N$ on $E$.\\

\textbf{Diagonalization Theorem} \cite[Theorem 2.8]{aupetitmoutontrace}: Let $a \in A$ be a nonzero maximal finite-rank element and denote by $\lambda_{1}, \ldots, \lambda_{n}$ its nonzero distinct spectral values. Then there exists $n$ orthogonal minimal projections $p_{1}, \ldots, p_{n}$ such that 
$$a = \lambda_{1}p_{1} + \cdots + \lambda_{n}p_{n}.$$

If $a \in \mathrm{Soc}\:A$ we define the \textit{trace} of $a$ as in \cite{aupetitmoutontrace} by
$$\mathrm{Tr}\,(a) = \sum_{\lambda \in \sigma (a)} \lambda m\left(\lambda, a\right),$$
where $m(\lambda,a)$ is the \emph{multiplicity of $a$ at $\lambda$}. A brief description of the notion of multiplicity in the abstract case goes as follows (for particular details one should consult \cite{aupetitmoutontrace}): Let $a \in \mathrm{Soc}\:A$, $\lambda\in\sigma(a)$ and let $B(\lambda,r)$ be an open disk centered at $\lambda$ such that $B(\lambda,r)$ contains no other points of $\sigma(a)$. It can be shown \cite[Theorem 2.4]{aupetitmoutontrace} that there exists an open ball, say $U\subseteq A$, centered at $\mathbf{1}$ such that $\#\left[\sigma(xa)\cap B(\lambda,r) \right]$ is constant as $x$ runs through $E(a)\cap U$. This constant integer is the multiplicity of $a$ at $\lambda$. It can also be shown that $m\left(\lambda,a\right) \geq 1$ and
\begin{equation}
\sum_{\alpha \in \sigma (a)} m(\alpha, a) = \left\{\begin{array}{cl} 1+\mathrm{rank}\,(a) & \mathrm{if}\;\,0 \in \sigma (a) \\
\mathrm{rank}\,(a) & \mathrm{if}\;\,0 \notin \sigma (a).
\end{array}\right.
\label{eq0}
\end{equation}

Let $\lambda \in \sigma (a)$ and suppose that $B(\lambda,2r)$ separates $\lambda$ from the rest of the spectrum of $a$. Let $f_{\lambda}$ be the holomorphic function which takes the value $1$ on $B(\lambda,r)$ and the value $0$ on $\mathbb{C} - \overline{B}(\lambda,r)$. If we now let $\Gamma_{0}$ be a smooth contour which surrounds $\sigma (a)$ and is included in the domain of $f_{\lambda}$, then
$$p\left(\lambda, a\right) = f_{\lambda} (a) = \frac{1}{2\pi i}\int_{\Gamma_{0}} f_{\lambda} (\alpha) \left(\alpha \mathbf{1}-a\right)^{-1}\,d\alpha$$
is referred to as the \textit{Riesz projection} associated with $a$ and $\lambda$. By the Holomorphic Functional Calculus, Riesz projections associated with $a$ and distinct spectral values are orthogonal and for $\lambda \neq 0$
\begin{equation}
p\left(\lambda, a\right) = \frac{a}{2\pi i} \int_{\Gamma_{0}} \frac{f_{\lambda}\left(\alpha\right)}{\alpha} \left(\alpha\mathbf{1} - a\right)^{-1}\,d\alpha \in aA. \label{eq1}
\end{equation}
The following results will also be useful: Let $a \in A$ have finite rank and let $\lambda_{1}, \ldots, \lambda_{n}$ be nonzero distinct elements of its spectrum. If 
$$p = p\left(\lambda_{1}, a\right) + \cdots + p\left(\lambda_{n}, a\right),$$ then by \cite[Theorem 2.6]{aupetitmoutontrace} we have
$$\mathrm{rank}\,(p) = m\left(\lambda_{1}, a\right) + \cdots + m\left(\lambda_{n}, a\right).$$
Moreover, $\mathrm{rank}\,(p) = m\left(1, p\right)$ \cite[Corollary 2.7]{aupetitmoutontrace}. It is customary to refer to $p$ here as the Riesz projection associated with $a$ and $\lambda_{1}, \ldots , \lambda_{n}$.\\

In the operator case, $A=B(X)$ (bounded linear operators on a Banach space $X$), the ``spectral" rank and trace both coincide with the respective classical operator definitions.

\section{Preliminaries}

Let $a \in \mathrm{Soc}\:A$. From (\ref{eq0}), it readily follows that
$$\sum_{\lambda \in \sigma' (a)} m\left(\lambda, a\right) \leq \mathrm{rank}\,(a).$$
Consequently, as observed in \cite{aupetitmoutontrace}, we have that
\begin{eqnarray*}
	\left|\mathrm{Tr}\,(a)\right| &=& \left|\sum_{\lambda \in \sigma (a)}\lambda m\left(\lambda, a\right)\right|
	= \left|\sum_{\lambda \in \sigma' (a)}\lambda m\left(\lambda, a\right)\right| \\
	& \leq & \sum_{\lambda \in \sigma' (a)}\left|\lambda \right| \cdot m\left(\lambda, a\right) 
	\leq  \sum_{\lambda \in \sigma' (a)}\rho (a) \cdot m\left(\lambda, a\right) \\
	& = & \rho (a)\cdot{\sum_{\lambda \in \sigma' (a)} m\left(\lambda, a\right)} 
	\leq  \rho (a)\cdot{\mathrm{rank}\,(a)}. 
\end{eqnarray*}
Furthermore, by \cite[Theorem 3.3(a)]{aupetitmoutontrace} it follows that $\mathrm{Tr}\, \left(x+y\right) = \mathrm{Tr}\,(x) + \mathrm{Tr}\,(y)$ for each $x, y \in \mathrm{Soc}\:A$. The next lemma shows that the trace is in fact a linear functional:

\begin{lemma}\label{1.1}
	Let $a$ be a finite-rank element of $A$ and let $\alpha \in \mathbb{C}-\left\{0 \right\}$. Then $m\left(\lambda, a\right) = m\left(\alpha\lambda, \alpha a\right)$ for each $\lambda \in \sigma' (a)$. Consequently, $\mathrm{Tr}\,\left(\alpha a\right) = \alpha \mathrm{Tr}\,(a)$ for each $\alpha \in \mathbb{C}$. 
\end{lemma}

\begin{proof}
	Set $\lambda_{0}=0$ and denote by $\lambda_{1}, \ldots, \lambda_{n}$ the distinct nonzero spectral values of $a$. Choose $r > 0$ so that the open disks $B\left(\lambda_{0}, r\right), B\left(\lambda_{1}, r\right), \ldots, B\left(\lambda_{n}, r\right)$ are all disjoint and $\alpha \in \mathbb{C}-\left\{0 \right\}$. By the Spectral Mapping Theorem \cite[Theorem 3.3.3(v)]{aupetit1991primer} it follows that $\sigma' \left(\alpha a\right) = \left\{\alpha\lambda_{1}, \ldots, \alpha\lambda_{n} \right\}$. Notice that
	$$B\left(\alpha\lambda_{0}, \left|\alpha\right| r\right),  \ldots, B\left(\alpha\lambda_{n}, \left|\alpha\right| r\right)$$
	are also all disjoint, and that, for each $i \in \left\{0, 1, \ldots, n \right\}$, we have $\beta \in B\left(\lambda_{i}, r\right)$ if and only if $\alpha \beta \in B\left(\alpha\lambda_{i}, \left|\alpha\right| r\right)$. Let $i \in \left\{ 1, \ldots, n \right\}$ be arbitrary but fixed. For $j \in \left\{ 1, \alpha\right\}$, let $U_{j}$ be an open disk centered at $\mathbf{1}$ so that $x \in U_{j} \cap E(ja)$ implies that
	$$\#\left(\sigma (jxa) \cap \Delta_{0}^{j}\right) = m \left(j\lambda_{i}, ja\right),$$
	where $\Delta_{0}^{j}$ is the interior of $\partial B\left(j\lambda_{i}, jr\right)$. Since $\#\sigma' (xa)  = \#\sigma' (\alpha xa)$ by the Spectral Mapping Theorem, and since $\mathrm{rank}\,(a) = \mathrm{rank}\,\left(\alpha a\right)$, it follows that
	$E(a) \subseteq E\left(\alpha a\right)$. Moreover, since $U_{1} \cap U_{\alpha} \cap E(a) \neq \emptyset$ by the density of $E(a)$, there exists an $x_{0} \in U_{1} \cap U_{\alpha} \cap E(a)$. But then $x_{0} \in U_{1} \cap E(a)$ and $x_{0} \in U_{\alpha} \cap E\left(\alpha a\right)$, so by the Spectral Mapping Theorem and our choice of contours we obtain that
	\begin{eqnarray*}
		m\left(\lambda_{i}, a\right) &=& \#\left(\sigma \left(x_{0}a\right) \cap \Delta_{0}^{1}\right) \\
		&=& \#\left(\sigma \left(\alpha x_{0}a\right) \cap \Delta_{0}^{\alpha}\right) \\
		&=& m\left(\alpha \lambda_{i}, \alpha a\right).
	\end{eqnarray*}
	Since $i$ was arbitrary, this completes the proof.
\end{proof}

In general, it is known that $A = \mathrm{Soc}\:A$ if and only if $A$ is finite-dimensional. Also, $\mathrm{Soc}\:A$ is a proper two-sided ideal whenever $A$ is infinite-dimensional. However, the following observation can be made concerning the dimension of $\mathrm{Soc}\:A$:

\begin{theorem}\label{2.1}
	The following are equivalent:
	\begin{itemize}
		\item[\textnormal{(a)}]
		$\mathrm{Soc}\:A$ is finite-dimensional.
		\item[\textnormal{(b)}]
		$\mathrm{Soc}\:A$ is closed.
		\item[\textnormal{(c)}]
		There exists a real number $c > 0$ such that $\left|\mathrm{Tr}\,(a)\right| \leq c \cdot \rho (a)$ for all $a \in \mathrm{Soc}\:A$.
	\end{itemize}
\end{theorem}

\begin{proof}
	The implication (a) $\Rightarrow$ (b) is clear. So assume that $\mathrm{Soc}\:A$ is closed. Recall that $\mathrm{Soc}\:A = \bigcup_{m=0}^{\infty}\mathcal{F}_{m}$. Moreover, from the lower semicontinuity of the rank on $\mathrm{Soc}\:A$ and the fact that $\mathrm{Soc}\;A$ is closed, it now follows that each $\mathcal{F}_{m}$ is closed. Thus, from Baire's Category Theorem, it follows that there is a smallest integer $n$ for which there is an open set $U \neq \emptyset$ in $\mathrm{Soc}\:A$ such that $U \subseteq \mathcal{F}_{n}$. Let $x \in U$ be arbitrary but fixed. Let $y \in \mathrm{Soc}\:A$ be arbitrary and consider the analytic function $f: \mathbb{C} \rightarrow A$ defined by $f\left(\lambda\right) = x - \lambda (x-y)$ for each $\lambda \in \mathbb{C}$. Since $U$ is open in $\mathrm{Soc}\:A$, and since $f(\lambda) \in \mathrm{Soc}\:A$ for each $\lambda \in \mathbb{C}$, it readily follows that there exists an $\epsilon > 0$ such that $\lambda \in B(0, \epsilon)$ implies that $f(\lambda) = x-\lambda (x-y) \in U$. Consequently, $\mathrm{rank}\,\left( f(\lambda)\right) \leq n$ for all $\lambda$ in a set with nonzero capacity. Thus, by the Scarcity Theorem for Rank, it follows that $\mathrm{rank}\,\left( f(\lambda)\right) \leq n$ for all $\lambda \in \mathbb{C}$. So, in particular, 
	$$\mathrm{rank}\,\left( f(1)\right) = \mathrm{rank}\,\left(x-(x-y)\right) =  \mathrm{rank}\,\left(y\right)  \leq n.$$
	Since $y \in \mathrm{Soc}\:A$ was arbitrary, it follows that $\mathrm{rank}\,\left(y\right)  \leq n$ for all $y \in \mathrm{Soc}\:A$. Consequently, it follows that
	$$\left|\mathrm{Tr}\,(a)\right| \leq \mathrm{rank}\,(a) \cdot \rho (a) \leq n \cdot \rho (a)$$
	for all $a \in \mathrm{Soc}\:A$. This shows that (b) $\Rightarrow$ (c). Suppose now that there exists a real number $c > 0$ such that $\left|\mathrm{Tr}\,(a)\right| \leq c \cdot \rho (a)$ for all $a \in \mathrm{Soc}\:A$. Let $k$ be any integer such that $k \geq c$. We claim that $\#\sigma'(y) \leq k$ for all $y \in \mathrm{Soc}\:A$: Suppose this is false. Then there exists an $x \in \mathrm{Soc}\:A$ such that $\#\sigma'(x) \geq k+1$. Let $\lambda_{1}, \ldots, \lambda_{k+1}$ be $k+1$ distinct nonzero spectral values of $x$ and let $p_{j} = p\left(\lambda_{j}, x\right)$ for each $j \in \left\{1, \ldots, k+1 \right\}$. Then $p = p_{1} + \cdots + p_{k+1}$ is a projection. Moreover, by our remarks in the introduction it follows that
	$$\mathrm{Tr}\,(p) = m\left(1, p\right) = \mathrm{rank}\,\left(p\right) = m\left(\lambda_{1}, x\right) + \cdots + m\left(\lambda_{k+1}, x\right) \geq k+1.$$
	But then
	$$\left|\mathrm{Tr}\,(p)\right| \geq k+1 > c = c \cdot \rho (p).$$
	This contradiction now proves our claim. Consequently, there exists a least integer $n$ such that $\mathrm{rank}\,\left(y\right)  \leq n$ for all $y \in \mathrm{Soc}\:A$. If $n = 0$, then $\mathrm{Soc}\:A = \left\{0 \right\}$ and we are done. So assume that $n \geq 1$. Since $n$ was the smallest integer with this property, it must be the case that $\mathrm{rank}\,(a) = n$ for some $a \in \mathrm{Soc}\:A$. Moreover, without loss of generality, we may assume that $a$ is a maximal finite-rank element. By the Diagonalization Theorem, there are $n$ orthogonal rank one projections $q_{1}, \ldots, q_{n}$ such that $a = \alpha_{1}q_{1} + \cdots +\alpha_{n}q_{n}$, where $\alpha_{1}, \ldots , \alpha_{n} \in \mathbb{C}-\left\{0 \right\}$. We claim that $\mathrm{Soc}\:A = \left(q_{1}+ \cdots +q_{n}\right)A$: Suppose not. Then there exists a $b \in \mathrm{Soc}\:A$ such that $b \notin \left(q_{1}+ \cdots +q_{n}\right)A$. Necessarily, $b \neq 0$. Let $u = \left(1-\left(q_{1} + \ldots +q_{n}\right)\right)b$. Then $u \in \mathrm{Soc}\:A$. Moreover, $u \neq 0$, for if $u=0$, then $b = q_{1}b + \cdots +q_{n}b$, which contradicts our choice of $b$. Finally, note that $q_{i}u=0$ for each $i \in \left\{1, \ldots , n \right\}$. Since $u \neq 0$, it follows that $\mathrm{rank}\,(u) = r$ for some integer $r$ with $1 \leq r \leq n$. Let $x \in A$ so that $\#\sigma' (ux) =r$ and let $\beta_{1}, \ldots , \beta_{r}$ be the distinct nonzero spectral values of $ux$. Now let $q$ be the Riesz projection associated with $ux$ and $\beta_{1}, \ldots , \beta_{r}$. Then $q \in \mathrm{Soc}\:A$. Moreover, by formula (\ref{eq1}) we have that $q_{i}q=0$ for each $i \in \left\{1, \ldots, n \right\}$. Consider the element
	$$v= 1q_{1} +2q_{2}+ \cdots +nq_{n} +(n+1)q,$$
	and note that $v \in \mathrm{Soc}\:A$. Since $q_{i}\left(i\mathbf{1}-v\right) = 0$ for each $i \in \left\{1, \ldots, n \right\}$, and since $\left((n+1)\mathbf{1}-v\right)q=0$, it follows that $\left\{1, \ldots, n+1 \right\} \subseteq \sigma (v)$. But this contradicts the fact that $\mathrm{rank}\,(v) \leq n$. Hence, $\mathrm{Soc}\:A = \left(q_{1}+ \cdots +q_{n}\right)A$ as claimed. By using a symmetric argument it can be shown that $\mathrm{Soc}\:A = A\left(q_{1}+ \cdots +q_{n}\right)$ as well. Hence, since every element of the socle is von Neumann regular (i.e. for each $a \in \mathrm{Soc}\:A$ there exists an $x \in \mathrm{Soc}\:A \subseteq A$ such that $a=axa$ \cite[Corollary 2.10]{aupetitmoutontrace}), we may infer that
	$$\mathrm{Soc}\:A = \sum_{i=1}^{n}\sum_{j=1}^{n}q_{i}Aq_{j}.$$
	Thus, since $\mathrm{dim}\,\left(q_{i}Aq_{j}\right) \leq 1$ for each $i, j \in \left\{1, \ldots, n \right\}$ (see \cite[Lemma 4.2]{puhltrace}), we have shown that (c) $\Rightarrow$ (a). This establishes the result.
\end{proof}

From the last part of the argument above it is actually possible to deduce a bit more. If $\mathrm{Soc}\:A$ is finite-dimensional, then it is possible to find a projection $p$, which is the finite sum of orthogonal rank one projections, such that $\mathrm{Soc}\:A = Ap = pA$. So, since every element of the socle is von Neumann regular, it follows that $\mathrm{Soc}\:A = pAp$. But $pAp$ is a closed semisimple subalgebra of $A$ with identity element $p$ (see \cite[Chapter 3, Exercise 6]{aupetit1991primer}). Hence, $\mathrm{Soc}\:A$ is a finite-dimensional semisimple Banach algebra with an identity element. Consequently, by the Wedderburn-Artin Theorem \cite[Theorem 2.1.2]{aupetit1991primer}, if $\mathrm{Soc}\:A$ is finite-dimensional, then it is isomorphic as an algebra to $M_{n_{1}} \left(\mathbb{C}\right) \oplus \cdots \oplus M_{n_{k}} \left(\mathbb{C}\right)$.\\

Let $p$ be a finite-rank projection of $A$. The subalgebra $pAp$ is very useful in the theory of rank, trace and determinant, primarily because of the following reasons:
\begin{equation}
\sigma_{pAp}' \left(pxp\right) = \sigma_{A}' \left(pxp\right)
\label{eq3}
\end{equation}
and
\begin{equation}
\mathrm{rank}_{pAp}(pxp) = \mathrm{rank}_{A}(pxp)
\label{eq4}
\end{equation}
for each $x \in A$. The proof of (\ref{eq3}) is not hard and (\ref{eq4}) is a consequence of (\ref{eq3}) and Jacobson's Lemma.\\

In order to prove our main results, some further preparation is needed:

\begin{lemma}\label{2.2}
	Let $f$ be a linear functional on $\mathrm{Soc}\:A$, and let $c > 0$ be a real number. Suppose that $\left|f(a)\right| \leq c \cdot \mathrm{rank}\,(a) \cdot \rho (a)$ for all $a \in \mathrm{Soc}\:A$. Then $f$ is continuous on $\mathcal{F}_{k}$ for each nonnegative integer $k$.
\end{lemma}

\begin{proof}
	Let $x \in \mathcal{F}_{k}$ be arbitrary, and suppose that $\left(x_{n}\right) \subseteq \mathcal{F}_{k} - \left\{x \right\}$ converges to $x$. We must show that $f\left(x_{n}\right) \rightarrow f(x)$ as $n \rightarrow \infty$: By the subadditivity of the rank it follows that
	$$\mathrm{rank}\,\left(x_{n} - x\right) \leq \mathrm{rank}\,\left(x_{n}\right) + \mathrm{rank}\,(x) \leq 2k$$
	for each integer $n \geq 1$. Thus, by linearity and the hypothesis on $f$ it follows that
	\begin{eqnarray*}
		\left|f\left(x_{n}\right) - f(x)\right| = \left|f\left(x_{n}-x\right)\right| & \leq & c \cdot \mathrm{rank}\,\left(x_{n}-x\right) \cdot \rho \left(x_{n}-x\right) \\
		& \leq & c \cdot 2k \cdot \left\|x_{n}-x\right\|.
	\end{eqnarray*}
	Since $x_{n} \rightarrow x$ as $n \rightarrow \infty$, we have the desired result. Hence, since $x \in \mathcal{F}_{k}$ was arbitrary, the lemma is proved. 
\end{proof}

\begin{theorem}\label{2.3}
	Let $f$ be a linear functional on $\mathrm{Soc}\:A$, and let $c > 0$ be a real number. Suppose that $\left|f(a)\right| \leq c \cdot \mathrm{rank}\,(a) \cdot \rho (a)$ for all $a \in \mathrm{Soc}\:A$. Then $f(ab) = f(ba)$ for all $a \in \mathrm{Soc}\:A$ and $b \in A$.
\end{theorem}

\begin{proof}
	Let $a \in \mathrm{Soc}\:A$ and $b \in A$ be arbitrary. Consider the analytic function $g: \mathbb{C} \rightarrow \mathrm{Soc}\:A$ given by $g(\lambda) = e^{\lambda b}ae^{-\lambda b}$. We claim that $\lambda \mapsto f\left(g(\lambda)\right)$ is an entire function: To prove our claim it will suffice to show that
	$$\lim_{\lambda \rightarrow \lambda_{0}} \frac{f\left(g(\lambda)\right) - f\left(g\left(\lambda_{0}\right)\right) }{\lambda - \lambda_{0}}$$
	exists for all $\lambda_{0} \in \mathbb{C}$. Let $\lambda_{0} \in \mathbb{C}$ be arbitrary, and let $\left(\alpha_{n}\right) \subseteq \mathbb{C}- \left\{\lambda_{0} \right\}$ be any sequence which converges to $\lambda_{0}$. Observe that
	$$\frac{f\left(g\left(\alpha_{n}\right)\right) - f\left(g\left(\lambda_{0}\right)\right) }{\alpha_{n} - \lambda_{0}}  =  f\left(\frac{g\left(\alpha_{n}\right) - g\left(\lambda_{0}\right)}{\alpha_{n} - \lambda_{0}} \right) $$
	for each integer $n \geq 1$, and that
	$$\mathrm{rank}\,\left(\frac{g\left(\alpha_{n}\right) - g\left(\lambda_{0}\right)}{\alpha_{n} - \lambda_{0}}\right) \leq 2 \cdot \mathrm{rank}\,(a)$$
	by the subadditivity and the properties of the rank mentioned earlier in the paper. Moreover, it can be shown that 
	$$g'\left(\lambda_{0}\right) = e^{\lambda_{0}b}bae^{-\lambda_{0}b} - e^{\lambda_{0}b}abe^{-\lambda_{0}b},$$
	so $g'\left(\lambda_{0}\right) \in \mathrm{Soc}\:A$ and $\mathrm{rank}\,\left(g'\left(\lambda_{0}\right)\right) \leq 2 \cdot \mathrm{rank}\,(a)$ as before. Hence, by Lemma \ref{2.2} it follows that
	$$\lim_{n \rightarrow \infty} \frac{f\left(g\left(\alpha_{n}\right)\right) - f\left(g\left(\lambda_{0}\right)\right) }{\alpha_{n} - \lambda_{0}} = f\left(g'\left(\lambda_{0}\right)\right).$$
	Thus, since $\lambda_{0} \in \mathbb{C}$ and $\left(\alpha_{n}\right)$ were arbitrary, this proves our claim. However, by hypothesis,
	\begin{eqnarray*}
		\left|f\left(g(\lambda)\right)\right| & \leq & c \cdot \mathrm{rank}\,\left(g(\lambda)\right) \cdot \rho \left(g(\lambda)\right) \\
		& = & c \cdot \mathrm{rank}\,(a) \cdot \rho (a)
	\end{eqnarray*}
	for each $\lambda \in \mathbb{C}$, where the equality sign follows from the properties of the rank and Jacobson's Lemma. So, since $\lambda \mapsto f\left(g(\lambda)\right)$ is entire and bounded, we may conclude by Liouville's Theorem that it must be constant. Hence, $0 = f\left(g'(0)\right) = f(ba-ab)$, so $f(ab) = f(ba)$ as desired.
\end{proof}

Since the trace is a linear functional on $\mathrm{Soc}\:A$ which satisfies
$$\left|\mathrm{Tr}\,(a)\right| \leq \mathrm{rank}\,(a) \cdot \rho (a) \;\,\mathrm{for\;all}\;\,a \in \mathrm{Soc}\:A,$$ 
Theorem \ref{2.3} readily gives the following:

\begin{corollary}\label{2.4}
	Let $a \in \mathrm{Soc}\:A$. Then $\mathrm{Tr}\,(ab) = \mathrm{Tr}\,(ba)$ for all $b \in A$.
\end{corollary}

\section{Trace Characterizations and Socle Identifications}

The next lemma can be obtained as a direct consequence of a deep result by K. Shoda (see \cite{shodathm}). However, we show that this result is not necessary for the development of our theory. 

\begin{lemma}\label{2.5}
	Let $A = M_{n} \left(\mathbb{C}\right)$. For $a, b \in A$ denote by $\left[a, b\right]$ the commutator $\left[a, b\right] =ab - ba$. Then $\mathrm{span}\left\{\left[a, b\right] : a, b \in A \right\} = \mathrm{Ker}\:\mathrm{Tr}$.
\end{lemma}

\begin{proof}
	By Corollary \ref{2.4} and the linearity of the trace it follows that 
	$$\mathrm{span}\left\{\left[a, b\right] : a, b \in A \right\} \subseteq \mathrm{Ker}\:\mathrm{Tr}.$$
	We prove the reverse containment. Notice that the traceless matrices are precisely those matrices which have arbitrary entries off the main diagonal and whose entries on the main diagonal sum to zero. Thus, if the $\left(i, i\right)$-entry is denoted by $\lambda_{i, i}$ for each $i \in \left\{1, \ldots, n \right\}$, then for the aforementioned traceless matrices we have $\lambda_{n, n} = - \left(\lambda_{1, 1} + \cdots + \lambda_{n-1, n-1}\right)$. We show that we can find a spanning set for the traceless matrices using only commutators. Let $e_{i, j}$ denote the matrix which has zeros at each entry except at the $\left(i, j\right)$-entry where it has a $1$. Consider any $e_{i, j}$ where $i \neq j$ (i.e. off the main diagonal). Then
	$$\left[e_{i,i}, e_{i, j}\right] = e_{i, i}e_{i, j} - e_{i, j}e_{i, i} = e_{i, j}.$$
	So we can generate all $0$ diagonal matrices with commutators. To deal with the diagonal consider $\left[e_{1, 2}, e_{2, 1}\right] = e_{1, 1} - e_{2, 2}\;\,;\;\,\left[e_{2, 3}, e_{3, 2}\right] = e_{2, 2} - e_{3, 3}\;\,;\;\,\left[e_{3, 4}, e_{4, 3}\right] = e_{3, 3} - e_{4, 4}\;\,;\;\,\cdots\;\,;\;\,\left[e_{n-1, n}, e_{n, n-1}\right] = e_{n-1, n-1} - e_{n, n}$. Thus, to get 
	$$\lambda_{1, 1}, \lambda_{2, 2}, \ldots, \lambda_{n-1, n-1}, -\left(\lambda_{1, 1}, \ldots, \lambda_{n-1, n-1}\right)$$
	as entries on the main diagonal just build the linear combination
	\begin{eqnarray*}
		&& \lambda_{1, 1}\left[e_{1, 2}, e_{2, 1}\right] + \left(\lambda_{1, 1} + \lambda_{2, 2}\right)\left[e_{2, 3}, e_{3, 2}\right] + \left(\lambda_{1, 1} + \lambda_{2, 2} + \lambda_{3, 3}\right)\left[e_{3, 4}, e_{4, 3}\right] \\
		& & + \cdots + \left(\lambda_{1, 1}, \ldots, \lambda_{n-1, n-1}\right)\left[e_{n-1, n}, e_{n, n-1}\right]. 
	\end{eqnarray*}
	This completes the proof.  
\end{proof}

\begin{lemma}\label{2.6}
	Let $x \in pAp = B$, where $p = p_{1} + \cdots + p_{n}$ with $p_{1}, \ldots, p_{n}$ orthogonal rank one projections of $A$. Then $\mathrm{Tr}_{A} (x) = \mathrm{Tr}_{B} (x)$.
\end{lemma}

\begin{proof}
	We have $x = pyp = \sum_{i=1}^{n}\sum_{j=1}^{n}p_{i}yp_{j}$. So, by the properties of the trace it will suffice to show that $\mathrm{Tr}_{A}\left(p_{i}yp_{i}\right) = \mathrm{Tr}_{B}\left(p_{i}yp_{i}\right)$ for each $i \in \left\{1, \ldots, n \right\}$. But this follows immediately from (\ref{eq3}) and (\ref{eq4}).
\end{proof}

\begin{theorem}\label{2.7}
	For any linear functional $f$ on $\mathrm{Soc}\:A$ we have that
	\begin{itemize}
		\item[\textnormal{(a)}]
		$f = \alpha \mathrm{Tr}$ for some $\alpha \in \mathbb{C}$,
		\item[\textnormal{(b)}]
		$f(ab) = f(ba)$ for all $a, b \in \mathrm{Soc}\:A$, and
		\item[\textnormal{(c)}]
		There exists a real number $c > 0$ such that $\left|f(a)\right| \leq c \cdot \rho (a)$ for all $a \in \mathrm{Soc}\:A$
	\end{itemize}
	are all equivalent if and only if $\mathrm{Soc}\:A \cong M_{n} \left(\mathbb{C}\right)$.
\end{theorem}

\begin{proof}
	For the reverse implication, by the remark preceding Lemma \ref{2.2}, and Lemma \ref{2.6}, it will suffice to prove that (a), (b) and (c) are all equivalent when $A = M_{n}\left(\mathbb{C}\right)$. Since $\mathrm{Soc}\:A$ is finite-dimensional, Theorem \ref{2.1} readily gives that (a) $\Rightarrow$ (c). Suppose now that $f$ satisfies (c). Then, in particular, $$\left|f(a)\right| \leq c \cdot \rho (a) \leq c \cdot \mathrm{rank}\,(a) \cdot \rho (a)$$ 
	for all $a \in \mathrm{Soc}\:A$. Hence, $f$ satisfies condition (b) by Theorem \ref{2.3}. This shows that (c) $\Rightarrow$ (b). Finally, if $f$ satisfies (b), then by Lemma \ref{2.5} $\mathrm{Ker}\:\mathrm{Tr} \subseteq \mathrm{Ker}\:f$. So $f$ is constant on the rank one projections of $A$, say $f(p) = \alpha$ for each rank one projection $p$. Let $a \in \mathrm{Soc}\:A-\left\{0 \right\}$ be arbitrary. By the density of $E(a)$ and the Diagonalization Theorem, it follows that $a = \lambda_{1}up_{1} + \cdots + \lambda_{m}up_{m}$, where $\lambda_{1}, \ldots, \lambda_{m} \in \mathbb{C}-\left\{0 \right\}$, $u \in G(A)$ and $p_{1}, \ldots, p_{m}$ are orthogonal rank one projections. Thus, since Corollary \ref{2.4} and the minimality of the $p_{i}$ gives $\lambda_{i}p_{i}up_{i} = \mathrm{Tr}\,\left(\lambda_{i}up_{i}\right)p_{i}$ for each $i \in \left\{1, \ldots, m  \right\}$, we obtain
	\begin{eqnarray*}
		f(a) & = & f\left(\lambda_{1}up_{1}\right) + \cdots + f\left(\lambda_{m}up_{m}\right) \\
		& = & f\left(\lambda_{1}p_{1}up_{1}\right) + \cdots + f\left(\lambda_{m}p_{m}up_{m}\right) \\
		& = & \mathrm{Tr}\,\left(\lambda_{1}up_{1}\right)f\left(p_{1}\right) + \cdots + \mathrm{Tr}\,\left(\lambda_{m}up_{m}\right)f\left(p_{m}\right) \\
		& = & \mathrm{Tr}\, \left(\lambda_{1}up_{1} + \cdots + \lambda_{m}up_{m}\right) \cdot \alpha  = \alpha \mathrm{Tr}\,(a).
	\end{eqnarray*}
	Thus, (b) $\Rightarrow$ (a), which establishes the desired equivalence. For the forward implication we note that $\left|\mathrm{Tr}\,(a)\right| \leq c \cdot \rho (a)$ for all $a \in \mathrm{Soc}\:A$, so $\mathrm{Soc}\:A$ is finite-dimensional by Theorem \ref{2.1}. Thus, by the remark preceding Lemma \ref{2.2} it follows that $\mathrm{Soc}\:A$ is isomorphic as an algebra to
	\begin{equation}
	B = M_{n_{1}}\left(\mathbb{C}\right) \oplus \cdots \oplus M_{n_{k}}\left(\mathbb{C}\right).
	\label{eq2}
	\end{equation}
	If $\mathrm{Soc}\:A = \left\{0 \right\}$, then the result is trivially true. So assume that $n_{i} \geq 1$ for each $i \in \left\{1, \ldots, n \right\}$. We claim that $B = M_{n_{1}}\left(\mathbb{C}\right)$, that is, that the direct sum above contains only one term: Suppose this is false, say $B$ appears as in (\ref{eq2}) with $k \geq 2$. Let $f : B \rightarrow \mathbb{C}$ be defined by $f\left(\left(a_{1}, \ldots, a_{k}\right)\right) = \mathrm{tr}\,\left(a_{1}\right)$, where $\mathrm{tr}$ denotes the trace in $M_{n_{1}}\left(\mathbb{C}\right)$. The linearity of $f$ on $B$ follows readily from that of $\mathrm{tr}$ on $M_{n_{1}}\left(\mathbb{C}\right)$. Also, $f \neq 0$ since $f\left(\left(\mathbf{1}, 0, \ldots, 0\right)\right) = n_{1}$. Moreover, if $a = \left(a_{1}, \ldots, a_{k}\right)$ and $b = \left(b_{1}, \ldots, b_{k}\right)$ are in $B$, then by Corollary \ref{2.4} it follows that
	$$f(ab) = \mathrm{tr}\,\left(a_{1}b_{1}\right) = \mathrm{tr}\,\left(b_{1}a_{1}\right) = f(ba).$$
	However, $f\left(\left(0, \mathbf{1}, 0, \ldots, 0\right)\right) = 0$, whereas $\mathrm{Tr}_{B}\left(\left(0, \mathbf{1}, 0, \ldots, 0\right)\right) \neq 0$. This is a contradiction, for it shows that $f \neq \alpha \mathrm{Tr}_{B}$ for all $\alpha \in \mathbb{C}$. So $B = M_{n_{1}}\left(\mathbb{C}\right)$ as advertised. This completes the proof.
\end{proof}

Remarkably it is possible to show that the condition that $f$ is constant on the rank one projections is enough to characterize the trace in general Banach algebras:

\begin{theorem}\label{2.8}
	Suppose that $\mathrm{Soc}\:A \neq \left\{0 \right\}$. For any linear functional $f$ on $\mathrm{Soc}\:A$ we have that $f = \alpha\mathrm{Tr}$ for some $\alpha \in \mathbb{C}$ if and only if $f$ is constant on the rank one projections of $A$.
\end{theorem}

\begin{proof}
	The forward implication follows from the definition of the trace. For the converse, suppose that $f$ is constant on the rank one projections of $A$. Let $p$ be any rank one projection of $A$ and let $x \in A$ be arbitrary. Then $\left(p + px - pxp\right)^{2} = p + px -pxp$. Moreover, $p + px -pxp \neq 0$, for otherwise $p = \left(p + px - pxp\right)p = 0$ which is absurd. Thus, since $p + px -pxp = p\left(p + px - pxp\right)$, it follows from the properties of the rank that $\mathrm{rank}\,\left(p + px - pxp\right) = 1$. Hence, by hypothesis we have $f(p) = f\left(p + px - pxp\right)$, and so $f(px) = f(pxp)$. This, together with the fact that $f$ is constant on the rank one projections, is enough to prove that $f = \alpha \mathrm{Tr}$ for some $\alpha \in \mathbb{C}$. Indeed, simply use the argument in the proof of Theorem \ref{2.7}.
\end{proof}

The next few results will be used to characterize those socles for which conditions (a) and (b) from Theorem \ref{2.7} are equivalent. This will lead to some new insights about the trace of finite rank operators.  

\begin{lemma}\label{2.9}
	There exists a collection of two-sided ideals $\left\{J_{p} : p \in \mathcal{P} \right\}$ such that every element of  $\mathrm{Soc}\:A$ can be written as a finite sum of members of the $J_{p}$. Each $J_{p}$ has the form
	$$J_{p} = \left\{ \sum_{j=1}^{n} x_{j}py_{j} : x_{j}, y_{j} \in A, n \geq 1\;\,\mathrm{an\;integer} \right\},$$
	where $p$ is a projection with $\mathrm{rank}\,(p) \leq 1$. Moreover, the two-sided ideals are pairwise orthogonal, that is, if $p, q \in \mathcal{P}$ with $p \neq q$, then
	$$J_{p}J_{q} = J_{q}J_{p} = \left\{0 \right\}.$$
\end{lemma}

\begin{proof}
	By definition a set $J_{p}$ as defined above is a two-sided ideal contained in $\mathrm{Soc}\:A$. Let $\mathcal{S}$ consist of all collections of these ideals such that the members in a collection are pairwise orthogonal. Then $\mathcal{S} \neq 0$ because (with $p = 0$) $\left\{\left\{0\right\} \right\} \in \mathcal{S}$. Partially order $\mathcal{S}$ by set containment. By Zorn's Lemma $\mathcal{S}$ has a maximal element, say $\mathpzc{M}$. We show that every element of $\mathrm{Soc}\:A$ can be written as a finite sum of elements each of which belongs to a member of $\mathpzc{M}$: By the density of the set $E(a)$ for $a \in \mathrm{Soc}\:A$ and the Diagonalization Theorem it suffices to show that we can do this for any minimal projection, say $q$. We claim that in fact $q$ belongs to some $J_{p} \in \mathpzc{M}$. Suppose this is not the case. Then for any particular $J_{p} \in \mathpzc{M}$, $q \notin J_{p}$. This implies that for each $x, y \in A$ we must have $xpyq = 0$, for otherwise, if there exists $x_{1}, y_{1} \in A$ such that $x_{1}py_{1}q \neq 0$, then by the minimality of $q$ we have $A\left(x_{1}py_{1}q\right) = Aq$, and so $zx_{1}py_{1}q = q$ for some $z \in A$. But then $q \in J_{p}$ which contradicts our assumption on $q$. So $xpyq = 0$ for all $x, y \in A$. Similarly, $qxpy = 0$ for all $x, y \in A$. If we now consider the ideal $J_{q}$, then we see that $J_{p}J_{q} = J_{q}J_{p} = \left\{0 \right\}$. But if this is true for each $J_{p}$ in $\mathpzc{M}$ then we have a contradiction with the maximality of $\mathpzc{M}$. This completes the proof.
\end{proof}

\begin{lemma}\label{2.10}
	Let $\mathpzc{M}$ be the collection of ideals obtained in the preceding lemma. If for any linear functional $f$ on $\mathrm{Soc}\:A$ we have that $f(ab) = f(ba)$ for each $a, b \in \mathrm{Soc}\:A$ implies $f = \alpha \mathrm{Tr}$ for some $\alpha \in \mathbb{C}$, then it must be the case that $\mathpzc{M}$ contains at most one nontrivial member.
\end{lemma}

\begin{proof}
	Suppose $\mathpzc{M}$ has more than one nontrivial member. Let $J_{p_{0}} \in \mathpzc{M}$ be arbitrary but not zero. Define $f$ on $\mathrm{Soc}\:A$ first implicitly as follows: $f\left(J_{p_{0}}\right) = \left\{0 \right\}$ and $f(a) = \mathrm{Tr}\,(a)$ if $a \notin J_{p_{0}}$ but $a \in J_{p}$ for some $J_{p} \in \mathpzc{M}$. In particular, $f(0)=0$. The next step is to show that $f$ is well-defined on the union of the members of $\mathpzc{M}$: It will suffice to show that distinct nonzero members of $\mathpzc{M}$ intersect only at $0$. If $0 \neq a \in J_{p} \cap J_{q}$, then we can write
	$$a = x_{1}py_{1} + \cdots +x_{n}py_{n} = u_{1}qv_{1} + \cdots +u_{k}qv_{k},$$
	and so for each $x \in A$, $xa \in J_{p} \cap J_{q}$. But notice now that $(xa)^{2} = 0$ since $J_{p}J_{q} = J_{q}J_{p} = \left\{0 \right\}$. Thus, $\sigma (xa) = \left\{0 \right\}$ so that $a \in \mathrm{Rad}\:A = \left\{0 \right\}$. This shows that $f$ is indeed well-defined on the union of the members of $\mathpzc{M}$. Now, by using a similar argument as above, it follows that every nonzero element of the socle can be written uniquely as a finite sum of nonzero elements from the members of $\mathpzc{M}$. So $f$ extends linearly to all of $\mathrm{Soc}\:A$. Now take any $a, b \in \mathrm{Soc}\:A$. Without loss of generality we can write $a = w_{p_{0}} + w_{p_{1}} + \cdots +w_{p_{n}}$ where $w_{p_{j}} \in J_{p_{j}}$ and $b = v_{p_{0}} + v_{q_{1}} + \cdots +v_{q_{n}}$ where $v_{p_{0}} \in J_{p_{0}}$ and $v_{q_{j}} \in J_{q_{j}}$. From the orthogonality of the members of $\mathpzc{M}$ and Corollary \ref{2.4} it follows that $f(ab) = f(ba)$. But it is clear that $f \neq \alpha \mathrm{Tr}$ for all $\alpha \in \mathbb{C}$ because $f\left(p_{0}\right) = 0$ and $f\left(q_{0}\right) = \mathrm{Tr}\,\left(q_{0}\right) = 1$ for some rank one projection $q_{0} \neq p_{0}$ (which exists by hypothesis). This contradiction completes the proof.  
\end{proof}

\begin{lemma}\label{2.11}
	If $pAp \not\cong M_{n} \left(\mathbb{C}\right)$ for some finite-rank projection $p$ of $A$, then $\mathpzc{M}$ contains at least two distinct members.
\end{lemma}

\begin{proof}
	By the Wedderburn-Artin Theorem \cite[Theorem 2.1.2]{aupetit1991primer} 
	$$pAp \cong M_{n_{1}}\left(\mathbb{C}\right) \oplus \cdots \oplus M_{n_{k}}\left(\mathbb{C}\right)$$
	and, by assumption, we may assume at least two nonzero terms in the direct sum. This means that we can find rank one projections, say $p$ and $q$ such that $pxq = 0$ and $qxp = 0$ for all $x \in A$. Suppose that $\mathpzc{M}$ has only one nontrivial member, say $J_{r}$, where $r$ is a rank one projection. If $\left(xry\right)p=0$ for all $x, y \in A$, then certainly $p \notin J_{r}$ which is contradictive to the fact that $\mathrm{Soc}\:A = J_{r}$. So $x_{1}ry_{1}p \neq 0$ for some $x_{1}, y_{1} \in A$. By minimality of $p$, we have $A\left(x_{1}ry_{1}p\right) = Ap$, and in turn the minimality of $r$ implies that $Ary_{1}p = Ap$. So $ry_{1}p = zp$ for some $z \in A$. Again minimality of $r$ gives $\left(ry_{1}p\right)A = rA$, and so $rA = zpA$. Thus, $r \in ApA$. So, since $\mathrm{Soc}\:A = J_{r}$, we have shown that
	$$J_{r} = \left\{ \sum_{j=1}^{n} x_{j}py_{j} : x_{j}, y_{j} \in A, n \geq 1\;\,\mathrm{an\;integer} \right\} = J_{p}.$$
	Similarly, it can be shown that $J_{r} = J_{q}$. But this is not possible since $J_{p} \cap J_{q} = \left\{0 \right\}$. Thus, the lemma is proved.
\end{proof}

\begin{theorem}\label{2.12}
	For any linear functional $f$ on $\mathrm{Soc}\:A$ we have that
	\begin{itemize}
		\item[\textnormal{(a)}]
		$f = \alpha \mathrm{Tr}$ for some $\alpha \in \mathbb{C}$, and
		\item[\textnormal{(b)}]
		$f(ab) = f(ba)$ for all $a, b \in \mathrm{Soc}\:A$
	\end{itemize}
	are equivalent if and only if $\mathrm{Soc}\:A$ is a minimal two-sided ideal.
\end{theorem}

\begin{proof}
	If $\mathrm{Soc}\:A = \left\{0 \right\}$, then the result trivially holds true. So assume that $\mathrm{Soc}\:A \neq \left\{0 \right\}$. If (a) and (b) are equivalent, then by Lemma \ref{2.10} $\mathrm{Soc}\:A = J_{p}$ for some rank one projection $p$. Let $I \subseteq \mathrm{Soc}\:A$ be a two-sided ideal and suppose that $a \in I- \left\{0 \right\}$. By the density of $E(a)$ and the Diagonalization Theorem, there exist a $u \in G(A)$, orthogonal rank one projections $p_{1}, \ldots, p_{n}$ and $\lambda_{1}, \ldots, \lambda_{n} \in \mathbb{C}-\left\{0 \right\}$ such that $ua = \lambda_{1}p_{1} + \cdots +\lambda_{n}p_{n}$. Consequently, if $x, y \in A$, then $\frac{1}{\lambda_{1}}xp_{1}uay = xp_{1}y \in I$. This shows that $J_{p_{1}} \subseteq I$. But since $p_{1} \in J_{p}$, it follows by the argument in the proof of Lemma \ref{2.11} that $J_{p_{1}} = J_{p}$. So $I = \mathrm{Soc}\:A$. This proves the forward implication. Conversely, if $\mathrm{Soc}\:A$ is minimal, then $\mathrm{Soc}\:A = J_{p}$ for some rank one projection $p$. Let $a \in \mathrm{Soc}\:A$ be arbitrary. Then $a = x_{1}py_{1} + \cdots +x_{n}py_{n}$ for some $x_{1}, \ldots, x_{n}, y_{1}, \ldots, y_{n} \in A$. So if $f$ satisfies (b), then, by Corollary \ref{2.4} and the fact that $p_{i}xp_{i} = \mathrm{Tr}\,\left(xp_{i}\right)p_{i}$ for all $x \in A$, it follows that 
	\begin{eqnarray*}
		f(a)& = & f\left(x_{1}py_{1}\right) + \cdots + f\left(x_{n}py_{n}\right) \\
		& = & f\left(py_{1}x_{1}p\right) + \cdots + f\left(py_{n}x_{n}p\right) \\
		& = & \mathrm{Tr}\,\left(y_{1}x_{1}p\right)f(p) + \cdots + \mathrm{Tr}\,\left(y_{n}x_{n}p\right) f(p) \\
		& = & \left[\mathrm{Tr}\,\left(x_{1}py_{1}\right) + \cdots +\mathrm{Tr}\,\left(x_{n}py_{n}\right)\right]f(p) \\
		& = & f(p)\mathrm{Tr}\,(a).
	\end{eqnarray*}
	This shows that (b) $\Rightarrow$ (a). The implication (a) $\Rightarrow$ (b) is of course a consequence of Corollary \ref{2.4}, so we have the result.
\end{proof}

\begin{theorem}\label{2.13}
	For each linear functional $f$ on $\mathrm{Soc}\:A$ we have that
	\begin{itemize}
		\item[\textnormal{(a)}]
		$f = \alpha \mathrm{Tr}$ for some $\alpha \in \mathbb{C}$, and
		\item[\textnormal{(b)}]
		$f(ab) = f(ba)$ for all $a, b \in \mathrm{Soc}\:A$
	\end{itemize}
	are equivalent if and only if $pAp \cong M_{n_{p}} \left(\mathbb{C}\right)$ for each finite-rank projection $p$ of $A$.
\end{theorem}

\begin{proof}
	If (a) and (b) are equivalent, then by Theorem \ref{2.12}  it follows that $\mathpzc{M}$ contains at most one nontrivial member. So Lemma \ref{2.11} gives the forward implication. For the converse, we note that (a) $\Rightarrow$ (b) by Corollary \ref{2.4}. So suppose that $f$ satisfies condition (b). Let $x \in \mathrm{Soc}\:A-\left\{0 \right\}$ be arbitrary. By the density of $E(x)$ and the Diagonalization Theorem, there exist a $u \in G(A)$, orthogonal rank one projections $p_{1}, \ldots, p_{n}$ and $\lambda_{1}, \ldots, \lambda_{n} \in \mathbb{C}-\left\{0 \right\}$ such that $x = \lambda_{1}up_{1} + \cdots +\lambda_{n}up_{n}$. Consequently, if $p = p_{1} + \cdots +p_{n}$, then $xp = x$. Thus, by hypothesis we have that $f(x) = f(xp) = f(pxp)$. Moreover, since $B = pAp \cong M_{n_{p}} \left(\mathbb{C}\right)$, it follows from Theorem \ref{2.7} that $\left. f\right|_{B} = \alpha_{B} \mathrm{Tr}_{B}$ for some $\alpha_{B} \in \mathbb{C}$. Hence, by Lemma \ref{2.6} and Corollary \ref{2.4}, we obtain
	\begin{eqnarray*}
		f(x) = f(pxp) & = & \alpha_{B} \mathrm{Tr}_{B} \left(pxp\right) = \alpha_{B} \mathrm{Tr}_{A} \left(pxp\right) \\
		& = & \alpha_{B} \mathrm{Tr}_{A} \left(xp\right) = \alpha_{B} \mathrm{Tr}_{A} \left(x\right).
	\end{eqnarray*}
	Since $x \in \mathrm{Soc}\:A-\left\{0 \right\}$ was arbitrary, this implies in particular that $\mathrm{Ker}\:\mathrm{Tr}_{A} \subseteq \mathrm{Ker}\:f$. Consequently, $f$ is constant on the rank one projections of $A$. So, by Theorem \ref{2.8} we have that (b) $\Rightarrow$ (a). This completes the proof.
\end{proof}

Although the trace on the socles of the Banach algebras in Theorem \ref{2.12} and Theorem \ref{2.13} is characterized, up to scalar multiples, as precisely those linear functionals on the socle which are $0$ on the commutators, it is possible to characterize the trace here by other properties. First, however, we show that these properties characterize the trace when $A = M_{n} \left(\mathbb{C}\right)$: 

\begin{lemma}\label{2.14}
	Let $A = M_{n} \left(\mathbb{C}\right)$. For any linear functional $f$ on $A$ the following are equivalent: 
	\item[\textnormal{(a)}]
	$f = \alpha \mathrm{Tr}$ for some $\alpha \in \mathbb{C}$. 
	\item[\textnormal{(b)}]
	$f(a) = 0$ for each nilpotent $a \in A$.
	\item[\textnormal{(c)}]
	For each $a \in A$, $f(a) = 0$ whenever $a^{2} = 0$.
\end{lemma}

\begin{proof}
	It is obvious that (a) $\Rightarrow$ (b) $\Rightarrow$ (c). So it will suffice to show that (c) $\Rightarrow$ (a): For each $i, j \in \left\{1, \ldots, n \right\}$, as before we let $e_{i, j}$ denote the matrix which has zeros at each entry except at the $\left(i, j\right)$-entry where it has a $1$. Consider any $e_{i, j}$ with $i \neq j$. In particular, $e_{i, j}^{2} = 0$. Consequently, for any $b \in A$ we have $\left(e^{\lambda b}e_{i, j}e^{-\lambda b}\right)^{2} = 0$ for all $\lambda \in \mathbb{C}$. So, by hypothesis $f\left(e^{\lambda b}e_{i, j}e^{-\lambda b}\right) = 0$ for all $\lambda \in \mathbb{C}$. Thus, $\lambda \mapsto f\left(e^{\lambda b}e_{i, j}e^{-\lambda b}\right)$ is a constant function from $\mathbb{C}$ into itself. Therefore, since $f$ is automatically continuous on $A$, it follows from a similar argument as the one used in Theorem \ref{2.3} that $f\left(be_{i, j} - e_{i, j}b\right) = 0$. Thus, $f\left(be_{i, j}\right) = f\left(e_{i, j}b\right)$ for all $b \in A$. Hence,
	$$f\left(e_{j, j}\right) = f\left(e_{j, i}e_{i, j}\right) = f\left(e_{i, j}e_{j, i}\right) = f\left(e_{i, i}\right).$$
	So, $f\left(e_{1, 1}\right) = f\left(e_{j, j}\right)$ for all $j \in \left\{1, \ldots, n\right\}$ and $f \left(e_{i, j}\right) = 0$ for $i \neq j$. Thus, if $a \in A$ is given by $a=\sum_{i=1}^{n}\sum_{j=1}^{n} \lambda_{i, j}e_{i, j}$, then
	$$f(a) = f\left(e_{1, 1}\right) \left( \lambda_{1, 1} + \cdots + \lambda_{n, n}\right) = f\left(e_{1, 1}\right)\mathrm{Tr}\,(a).$$
	Hence, the lemma holds true.
\end{proof}

\begin{theorem}\label{2.15}
	Suppose that $\mathrm{Soc}\:A$ is a minimal two-sided ideal of $A$. Then for every linear functional $f$ on $\mathrm{Soc}\:A$ the following are equivalent:
	\begin{itemize}
		\item[\textnormal{(a)}]
		$f = \alpha \mathrm{Tr}$ for some $\alpha \in \mathbb{C}$. 
		\item[\textnormal{(b)}]
		$f(a) = 0$ for each nilpotent $a \in \mathrm{Soc}\:A$.
		\item[\textnormal{(c)}]
		For each $a \in \mathrm{Soc}\:A$, $f(a) = 0$ whenever $a^{2} = 0$.
		\item[\textnormal{(d)}]
		There exists a real number $c > 0$ such that $\left|f(a)\right| \leq c \cdot \mathrm{rank}\,(a) \cdot \rho (a)$ for all $a \in \mathrm{Soc}\:A$.
	\end{itemize}
\end{theorem}

\begin{proof}
	It is clear that (a) $\Rightarrow$ (b) $\Rightarrow$ (c) and that (a) $\Rightarrow$ (d). Moreover, by Theorem \ref{2.12} and Theorem \ref{2.3} it follows that (d) $\Rightarrow$ (a). It will therefore suffice to show that (c) $\Rightarrow$ (a): Let $x \in \mathrm{Soc}\:A - \left\{0 \right\}$ be arbitrary. As in the proof of Theorem \ref{2.13} we can find a finite-rank projection $p$ such that $xp = x$. Since $\left(xp - pxp\right)^{2} = 0$, the assumption on $f$ gives $f(x) = f(xp) = f(pxp)$. Moreover, by Theorem \ref{2.12} and Theorem \ref{2.13} we have $B = pAp \cong M_{n_{p}} \left(\mathbb{C}\right)$. So, by Lemma \ref{2.14} and the argument used in the proof of Theorem \ref{2.13} we have the desired implication.
\end{proof}

As we will see after the next lemma, for any Banach space $X$, $\mathrm{Soc}\:B(X)$ is a minimal two-sided ideal.

\begin{lemma}\label{2.16}
	Let $P, Q$ be rank one projections of $B(X)$, where $X$ is some Banach space. Then there exist $S, T \in B(X)$ such that $P-Q = ST-TS$ and $S$ and $T$ are both of rank one.
\end{lemma}

\begin{proof}
	By hypothesis there exist $x, y \in X - \left\{0 \right\}$ such that $Pu = f(u)x$ and $Qu = g(u)y$ for each $u \in X$, where $f, g \in X'$ (the dual space of $X$). Moreover, since $P^{2} = P$ and $Q^{2} = Q$, it follows that $f(x)x = x$ and $g(y)y = y$. Hence, since $x, y \in X - \left\{0 \right\}$, it follows that $f(x) = g(y) = 1$. Let $S: X \rightarrow X$ and $T: X \rightarrow X$ be defined by $Su = g(u)x$ and $Tu = f(u)y$ for each $u \in X$. Then $S, T \in B(X)$ and $S$ and $T$ are both of rank one. Furthermore,
	$$\left(ST\right)(u) = S\left(f(u)y\right) = f(u)g(y)x = f(u)x = Pu$$
	and
	$$\left(TS\right)(u) = T\left(g(u)x\right) = g(u)f(x)y = g(u)y = Qu$$
	for each $u \in X$. Thus, $P-Q = ST-TS$ as desired. 
\end{proof}

\begin{theorem}\label{2.17}
	Let $A = B(X)$ for some Banach space $X$. Then for every linear functional $f$ on $\mathrm{Soc}\:A$ the following are equivalent:
	\begin{itemize}
		\item[\textnormal{(a)}]
		$f = \alpha \mathrm{Tr}$ for some $\alpha \in \mathbb{C}$.
		\item[\textnormal{(b)}]
		$f(ab) = f(ba)$ for all $a, b \in \mathrm{Soc}\:A$.
		\item[\textnormal{(c)}]
		There exists a real number $c > 0$ such that $\left|f(a)\right| \leq c \cdot \mathrm{rank}\,(a) \cdot \rho (a)$ for all $a \in \mathrm{Soc}\:A$.
		\item[\textnormal{(d)}]
		$f(a)=0$ for each nilpotent $a \in \mathrm{Soc}\:A$.
		\item[\textnormal{(e)}]
		For each $a \in \mathrm{Soc}\:A$, $f(a) = 0$ whenever $a^{2} = 0$.
	\end{itemize}
\end{theorem}

\begin{proof}
	By Theorem \ref{2.12} and Theorem \ref{2.15} it will suffice to show that (a) and (b) are equivalent. By Corollary \ref{2.4} we know that (a) $\Rightarrow$ (b). Conversely, by Lemma \ref{2.16} it follows that $f$ is constant on the rank one projections. So by Theorem \ref{2.8} it follows that (b) $\Rightarrow$ (a), establishing the result.
\end{proof}

\bibliographystyle{amsplain}
\bibliography{Spectral}

\end{document}